\documentclass[11pt,reqno]{amsart}
\usepackage{amsmath,mathtools,amsthm,amssymb}
\usepackage{graphicx}
\usepackage{epsfig,epsf,psfrag}
\usepackage{amsfonts}
\usepackage{setspace}
\usepackage{color}
\usepackage[font=small]{caption}
\usepackage[font=footnotesize]{subcaption}
\usepackage[top=1in, bottom=1in, left=1.2in, right=1.2in]{geometry}
\usepackage{afterpage}

\usepackage{isomath}

\newtheorem{theorem}{Theorem}[section]
\newtheorem{corollary}[theorem]{Corollary}
\newtheorem{lemma}[theorem]{Lemma}

\newtheorem{remark}[theorem]{Remark}

\title{Singularity swapping method for nearly singular integrals based on trapezoidal rule}

\author{Gang Bao, Wenmao Hua, Jun Lai, Jinrui Zhang} 
\address{School of Mathematical Sciences, Zhejiang University,
	Hangzhou, Zhejiang 310027, China}
\email{baog@zju.edu.cn, huawenmao@zju.edu.cn, laijun6@zju.edu.cn, zhangjinrui@zju.edu.cn}
\thanks{JL was supported by the ``Xiaomi Young Scholars" program from Xiaomi Foundation.}


\begin{document}
	
	\maketitle
	
	\begin{abstract}
		Accurate evaluation of nearly singular integrals plays an important role in many boundary integral equation based numerical methods. In this paper, we propose a variant of singularity swapping method to accurately evaluate the layer potentials for arbitrarily close targets. Our method is based on the global trapezoidal rule and trigonometric interpolation, resulting in an explicit quadrature formula.  The  method achieves spectral accuracy for nearly singular integrals on closed analytic  curves. In order to extract the singularity from the complexified distance function, an efficient root finding method is proposed based on contour integration. Through the change of variables, we also extend the quadrature method to integrals on the piecewise analytic curves.  Numerical examples for Laplace's and Helmholtz equations show that high order accuracy can be achieved for arbitrarily close field evaluation.
	\end{abstract}

	\section{Introduction}
	The method of boundary integral equations (BIE) has been widely used in the numerical computation of Laplace and wave scattering problems, examples including Helmholtz equations in acoustics\cite{JL2014}, Maxwell equations in electromagnetics\cite{LAI2019152}, and Navier equations in elasticity\cite{DLL21}.  The appealing advantages of BIE are that they automatically satisfy the radiation condition at infinity for wave scattering in unbounded region, and reduce the dimension of computational domain by one when the medium coefficient or wavenumber is piecewise constant. 
	
	
	However, these advantages do not come for free. As the boundary integral equations are constructed through layer potentials, one of the major difficulties encountered in the numerical application of BIE is the evaluation of singular or nearly singular integrals\cite{KLOCKNER2013332}, which, in the two dimensional setting, is typically given in the form of  
	\begin{equation}\label{layer potential}
		\int_{\Gamma}K(x,y)\phi(y)ds(y),\quad x\in\mathbb R^2,
	\end{equation}
	where $\Gamma$ is assumed to be a  closed analytic curve in $\mathbb{R}^2$. The density function $\phi$ is assumed to be analytic on $\Gamma$, and  $K$ is the kernel function that has a singularity as $x\to y$.  Generally, the kernel function $K(x,y)$ is of logarithmic form\cite{Klinteberg2020}
	\begin{equation}\label{logarithmic}
		K(x,y)=\log|x-y|k(x,y),
	\end{equation}
	(here we use $\mathit{log}$ to denote the natural logarithm) or of power-law form
	\begin{equation}\label{power-law form}
		K(x,y)=\frac{k(x,y)}{|x-y|^{2m}},
	\end{equation}
	where $m$ is a positive integer  and $k(x,y)$ is an analytic function in the neighborhood of $\Gamma$. Extension of integral \eqref{layer potential} to piecewise analytic curve will also be discussed later on. Applications of integral \eqref{layer potential} include the single and double layer potentials for Laplace's and Helmholtz equations. 
	
	When the target point $x$ is far away from the boundary, the integrand in \eqref{layer potential} is smooth, and high-order quadratures can be easily obtained by standard methods, such as Gauss-Legendre quadrature or trapezoidal rule. In particular, when $\Gamma$ is a closed analytic curve, trapezoidal rule gives spectral accuracy\cite{Trefethen2014} for smooth integrals, as reviewed in the following Section \ref{interpolatory quadrature}. On the other hand, when $x$ is on the boundary $\Gamma$,  integral \eqref{layer potential} becomes singular and a well-known method to evaluate it is through singularity extraction. Details can be found in \cite{Kress2010}. Difficulties arise when $x$ is near the curve, but not necessary on the curve, in which case the integral $\eqref{layer potential}$ is called the close evaluation problem. This can happen in many applications when the near field interaction is required, such as the multi-particle scattering in electromagnetics~\cite{Lai2015194}, or the evaluation of pressure gradient due to a force on a moving boundary in the Stokes equation \cite{Barnett2015}. It also happens for the on-boundary field evaluation when $\Gamma$ is wiggly shaped so that far away points in the parameter space may be very close geometrically. For the close evaluation problem, due to the sharply peaked kernel in the layer potential, standard quadrature rule produces an $O(1)$ error. 
	
	Due to the wide range applications of BIE, the close evaluation problem has been extensively studied in the literature. In \cite{Schwab1999OnTE}, a recursive method for the evaluation of potentials and their derivatives near and up to the boundary surface based on the asymptotic pseudo-homogeneous kernel expansions was constructed. In \cite{doi:10.1137/S0036142999362845}, a regularization type method was proposed by first regularizing the singularity, evaluating the regularized integral by a standard quadrature rule, and then adding a correction term to eliminate the error. In \cite{1514571}, the authors found a change of variable technique that canceled the principal singularity and then applied the Gauss-Legendre quadrature to the resulted integral. In \cite{HELSING20082899}, a barycentric interpolation for the Laplace's equation has been proposed to evaluate the nearly singular integral. In \cite{KLOCKNER2013332} and \cite{doi:10.1137/120900253}, quadrature by expansions(QBX) method has been developed based on the continuity of singular integrals near the boundary. It can efficiently evaluate the nearly singular integrals and is compatible with the fast multipole method(FMM). Other methods include singularity subtraction techniques \cite{CARVALHO2018327},  density subtraction techniques \cite{PEREZARANCIBIA2019411,2020arXiv200707965C}, special purpose quadratures \cite{2022arXiv220304854I}, etc. Although some of the methods are very efficient in the numerical implementation, the corrected integrands in many techniques remain singular or nearly singular in the sense that they contain large or unbounded higher order derivatives, resulting in only algebraic accuracy.
	
	Recently, Klinteberg and Barnett\cite{Klinteberg2020} proposed a singularity swapping method that  swaps the target singularity from the physical space into the parameter space, and integrates the transformed integrand by composite Gauss-Legendre quadrature and polynomial interpolation. It naturally extends the singularity subtraction technique from the on-boundary field evaluation to the near field evaluation, and achieves a uniform high-order accuracy for arbitrarily close targets. However, its convergence rate is still algebraic for integrals on smooth closed curves. In this paper, we introduce a variant of that method by using global trapezoidal rule and trigonometric interpolation for the transformed integrand. It results in an explicit quadrature formula with exponential convergence for integrals defined on closed analytic curves. We also propose a quadrature method to find the complex singularity efficiently based on the modified trapezoidal rule. The method is extended to integrals defined on piecewise analytic curves by substituting a suitable new variable and then applying the singularity swapping method for the transformed integral. Overall, our algorithm is simple, elegant and yet efficient, in the sense that it only needs to slightly modify the classic trapezoidal rule and still achieves rapid convergence rate.

	The remainder of this paper is organized as follows. Section \ref{interpolatory quadrature} presents an interpolatory quadrature on equidistant nodes for periodic nearly singular integrals. Section \ref{singularity swapping} presents a singularity swapping method that transforms the integral \eqref{layer potential} into the form discussed in Section \ref{interpolatory quadrature}. Applications to Laplace and Helmholtz layer potentials are also discussed in this section. The extension to non-periodic nearly singular integrals is given in Section \ref{section_nonperiodic_integrand}. We demonstrate the effectiveness of the proposed method by numerical examples presented in Section \ref{numerical examples}. Section \ref{conclusions} concludes our paper.
	

	\section{Periodic nearly singular integrals}\label{interpolatory quadrature}
	In this section, we propose a modified trapezoidal rule to accurately evaluate the periodic nearly singular integrals based on the error analysis of classic trapezoidal rule. It is well known that for an integral with $C^2$ integrand, the classic trapezoidal rule only gives second order accuracy. However, when the integrand is periodic and belongs to $C^{2k}$, Euler-Macularin formula shows it can achieve an accuracy of order $2k$. If the function is further assumed to be analytic,  trapezoidal rule gives spectral accuracy\cite{Trefethen2014}. Therefore, the convergence rate of trapezoidal rule highly depends on the smoothness of the integrand.  Here we briefly review the spectral convergence property of trapezoidal rule for periodic smooth integrals and then extend it to the nearly singular case. 
	
	\subsection{Review of trapezoidal rule}
	In this subsection, we are mainly concerned with the integral
	\begin{equation}\label{original integral}
		\int_0^{2\pi}f(t)dt,
	\end{equation}
	where $f$ is an analytic and periodic function on $[0,2\pi]$. Rewrite $f(t)=u(z)$ with $z=e^{it}$. The original integral \eqref{original integral} becomes
	\begin{equation}\label{integral}
		I(u) = \int_0^{2\pi}u(e^{it})dt = \int_{|z|=1}u(z)\frac{dz}{iz}.
	\end{equation}
	Throughout this paper, a complex contour such as $|z|=1$ is always to be understood as traversed once in the counterclockwise direction. Since $f$ is analytic on $[0,2\pi]$, by analytic extension, there exists $r>1$ such that $u$ is analytic  in the region $r^{-1}\le |z|\le r$. For any positive integer $n$, the  approximation to $I(u)$ using trapezoidal rule with $n$ quadrature nodes is simply given by
	\begin{equation}\label{trapezoidal rule}
		T_n(u) = \frac{2\pi}{n}\sum_{k=1}^nu(e^{2\pi ik/n}).
	\end{equation}
	To analyze the error of $T_n(u)-I(u)$, note that the function $\frac{1}{z^n-1}\frac{1}{iz}$
	has $n$ simple poles at $e^{2\pi ik/n}$, $k=1,2,\cdots, n$, with residues equal to $1/(in)$. By the residue theorem, it holds
	\begin{equation}\label{trapezoidal rule integral form}
		T_n(u) = \int_{|z|=r}\frac{u(z)}{z^n-1}\frac{dz}{iz} - \int_{|z|=r^{-1}}\frac{u(z)}{z^n-1}\frac{dz}{iz}.
	\end{equation}
	Since $u$ is analytic in $r^{-1}\le |z|\le r$, we can change $|z|=1$ to the circle $|z|=r^{-1}$ in equation \eqref{integral} without changing the value. Thus, combining \eqref{integral} and \eqref{trapezoidal rule integral form} gives the representation of $T_n(u)-I(u)$ as an integral,
	\begin{equation}\label{trapezoidal rule error integral form}
		T_n(u)-I(u) = \int_{|z|=r}\frac{u(z)}{z^n-1}\frac{dz}{iz} + \int_{|z|=r^{-1}}\frac{u(z)}{z^{-n}-1}\frac{dz}{iz}.
	\end{equation}
	Using \eqref{trapezoidal rule error integral form} and the boundedness of $u(z)$, we obtain the exponential convergence of trapezoidal rule. 
	\begin{lemma}\label{lemma1}
		If	$u$ is analytic in the annulus $r^{-1}\leqslant|z|\leqslant r$ for some $r>1$, then
		\begin{equation}\label{trapezoidal rule error asymptotic form}
			T_n(u)-I(u) = O(r^{-n}), \mbox{ when } \quad n\to\infty.
		\end{equation}
	\end{lemma} 
	
	Based on  \eqref{trapezoidal rule error integral form}, one also finds that  $T_n(u)=I(u)$ for $u(z)=z^{-n+1},\dots,z^{n-1}$ and $z^n-z^{-n}$. Therefore, the trapezoidal rule can be treated as a Gaussian quadrature on the unit circle $|z|=1$\cite{Trefethen2014}.
	
	\subsection{Nearly singular integrals}\label{subsection nearly singular integral}
	We now turn to the evaluation of nearly singular integral defined on $|z|=1$, which is given in the form of
	\begin{equation}\label{nearly singular integral}
		I(Ku) = \int_{|z|=1}K(z)u(z)\frac{dz}{iz},
	\end{equation}
	where $u$ is analytic in the annulus $R^{-1}\leqslant|z|\leqslant R$ for some $R>1$ and $K$ is analytic in a smaller annulus $r^{-1}\leqslant|z|\leqslant r$ with $1<r<R$. 
	
	Presumably, $r$ is very close to $1$, which means $K$ is badly behaved on the unit circle. According to \eqref{trapezoidal rule error asymptotic form}, trapezoidal rule gives very poor convergence rate that is on the order of $O(r^{-n})$.  To overcome this difficulty, one way is to adaptively place very dense quadrature nodes on the unit circle that are close to the singularity of $K$. Another way is to numerically recalculate the weights of trapezoidal rule based on the variance of Euler-Macularin formula so that the singularity of $K$ is taken into account\cite{Alpert1999}. Both ways will only lead to an algebraic convergence of \eqref{nearly singular integral}. Here we propose a  modified trapezoidal rule,  in which the quadrature weights can be explicitly obtained and the error still keeps exponential decay. The main idea is to replace $K$ with its Laurent expansion $K_n$ and then apply the trapezoidal rule to $K_n u$. 
	
	More precisely, let
	\begin{equation}\label{taylor expansion}
		K_n(z) = \frac12c_{-n}z^{-n} + \sum_{k=-n+1}^{n-1}c_kz^k + \frac12c_nz^n
	\end{equation}
	be the $n$th order Laurent expansion of $K$ at zero, where
	\begin{equation}\label{taylor coefficients}
		c_k = \frac1{2\pi i}\int_{|z|=1}K(z)z^{-k-1}dz.
	\end{equation}
	Note that the $n$-th term of the Laurent series \eqref{taylor expansion} is halved, as well as the $-n$-th term, which is not mandatory but we do this for symmetry and consistency with the classical trigonometric interpolation, as shown later. The modified trapezoidal rule for integral \eqref{nearly singular integral} is simply
	\begin{eqnarray}\label{modifedtrap}
		I(Ku)\approx T_{2n}(K_nu) = \frac{\pi}{n}\sum_{k=1}^{2n}(K_nu)(e^{\pi ik/n}).
	\end{eqnarray}
	Now let us analyze the error of formula \eqref{modifedtrap}. It follows from equation \eqref{trapezoidal rule integral form}  that
	\begin{equation}
		T_{2n}(K_nu) = \int_{|z|=R}\frac{K_n(z)}{z^{2n}-1}u(z)\frac{dz}{iz} - \int_{|z|=R^{-1}}\frac{K_n(z)}{z^{2n}-1}u(z)\frac{dz}{iz}.
	\end{equation}
	Write $K=K_n^-+K_n+K_n^+$, where
	\[K_n^-(z) = \sum_{k=-\infty}^{-n-1}c_kz^k + \frac12c_{-n}z^{-n},\quad K_n^+(z) = \frac12c_nz^n + \sum_{k=n+1}^{\infty}c_kz^k.\]
	Using the fact that $K_n^-u$ is analytic in $r^{-1}\leqslant|z|\leqslant R$ and $K_nu+K_n^+u$ is analytic in $R^{-1}\leqslant|z|\leqslant r$, it holds
	\begin{align}\label{interpolatory quadrature error integral form}
		& T_{2n}(K_nu)-I(Ku) \notag\\
		=& \int_{|z|=R}\left(\frac{K_n(z)}{z^{2n}-1}-K_n^-(z)\right)u(z)\frac{dz}{iz} \notag \\&- \int_{|z|=R^{-1}}\left(\frac{K_n(z)}{z^{2n}-1}+K_n(z)+K_n^+(z)\right)u(z)\frac{dz}{iz} \notag\\
		=& \int_{|z|=R}\left(\frac{K_n(z)}{z^{2n}-1}-K_n^-(z)\right)u(z)\frac{dz}{iz} + \int_{|z|=R^{-1}}\left(\frac{K_n(z)}{z^{-2n}-1}-K_n^+(z)\right)u(z)\frac{dz}{iz}.
	\end{align}
	From the Cauchy's estimate
	\begin{equation}\label{cauchy's estimate}
		|c_k| \leqslant \|K\|_\infty r^{-|k|}, \quad \mbox{ with } \|K\|_\infty:=\sup_{r^{-1}\leqslant|z|\leqslant r}|K(z)|,
	\end{equation}
	we have the following results on $|z|=R$
	\begin{eqnarray}
		\left|\frac{K_n(z)}{z^{2n}-1}\right| \leqslant \frac{\sum_{k=-n}^n\|K\|_\infty r^{-k}R^k}{R^{2n}-1}\leqslant \frac{\|K\|_\infty }{R^{2n}-1}\frac{(\frac Rr)^{n+1}}{\frac Rr-1},
		\label{first term} \\
		|K_n^-(z)| \leqslant \sum_{k=-\infty}^{-n}\|K\|_\infty r^kR^k = \frac{\|K\|_\infty (Rr)^{-n+1}}{Rr-1}.\label{second term}
	\end{eqnarray}
	Same estimates hold for $\left|\frac{K_n(z)}{z^{-2n}-1}\right|$ and $|K_n^+(z)|$ on $|z|=R^{-1}$. Combining \eqref{interpolatory quadrature error integral form}, \eqref{first term} and \eqref{second term} gives
	\begin{equation}\label{interpolatory quadrature error asymptotic form}
		|T_{2n}(K_nu)-I(Ku)| \leqslant 4\pi\|K\|_\infty\|u\|_\infty\left(\frac{(\frac Rr)^{n+1}}{(\frac Rr-1)(R^{2n}-1)}+\frac{(Rr)^{-n+1}}{Rr-1}\right).
	\end{equation}
	Furthermore, equation \eqref{interpolatory quadrature error integral form}  implies $T_{2n}(K_nu)=I(Ku)$ for $u(z)=z^{-n+1},\dots,z^{n-1}$ and $z^n+z^{-n}$.
	
	To summarize, we have the following result.
	\begin{theorem}\label{modifiedtrap}
		Assume	$u$ is analytic in the annulus $R^{-1}\leqslant|z|\leqslant R$ for some $R>1$ and $K$ is analytic in a smaller annulus $r^{-1}\leqslant|z|\leqslant r$ with $1<r<R$. It holds
		\begin{eqnarray}\label{singquaderror}
			T_{2n}(K_nu)-I(Ku) = O((Rr)^{-n}), \quad n\to\infty.
		\end{eqnarray}
	\end{theorem}
	
	Note that the presented method can be regarded as interpolating $u$ at $e^{2\pi ik/(2n)}$ into the form of $a_{-n+1}z^{-n+1}+\cdots+a_{n-1}z^{n-1}+a_n(z^n+z^{-n})$ and then integrating analytically, so it can be treated as an interpolatory quadrature. Compared to the original trapezoidal rule, it only has half of the polynomial order of accuracy, as we have preassigned $2n$ quadrature nodes at $e^{2\pi ik/(2n)}$. It is possible to construct a generalized Gaussian quadrature that exactly integrates $2n$ functions using only $n$ function evaluations \cite{Daan2009}, but that would require solving nonlinear equations and the corresponding quadrature nodes will vary with $K$ \cite{BGR2010}. From that perspective, the presented method of fixed quadrature nodes is more convenient for practical computation. 
	
	To illustrate the application of the modified trapezoidal rule \eqref{modifedtrap}, we give the explicit expressions for the nearly singular integrals of some typical kernel functions, which can be used to evaluate the near field of single and double layer potentials for Laplace's and Helmholtz equations.
	
	\begin{corollary}\label{coroll1}
		For the nearly singular integral of
		\begin{eqnarray}\label{logintegral}
			\log\left(4\sin\frac{t-t_0}{2}\sin\frac{t-\overline{t_0}}{2}\right)
		\end{eqnarray} with $t\in [0,2\pi]$ and $t_0\in\mathbb{C}$ close to the line segment $[0,2\pi]$. It holds
		\begin{equation}\label{log quadrature}
			\int_0^{2\pi}\log\left(4\sin\frac{t-t_0}{2}\sin\frac{t-\overline{t_0}}{2}\right)f(t)dt \approx \frac{\pi}{n}\sum_{j=0}^{2n-1}R_n(t_0,j\pi/n)f(j\pi/n).
		\end{equation}
		where
		\[R_n(t_0,t) = |\Im t_0| - 2\sum_{k=1}^{n-1}\frac{e^{-k|\Im t_0|}}{k}\cos k(t-\Re t_0) - \frac{e^{-n|\Im t_0|}}{n}\cos n(t-\Re t_0)\]
		is $n$-th order Laurent expansion of \eqref{logintegral}. 
	\end{corollary}
	When $\Im t_0=0$, formula \eqref{log quadrature} reduces to the well known Kress quadrature~\cite{Kress2010}:
	\begin{equation}\label{Kress quadrature}
		\int_0^{2\pi}\log\left(4\sin^2\frac{t-t_0}{2}\right)f(t)dt \approx \frac{\pi}{n}\sum_{j=0}^{2n-1}R_n(t_0,j\pi/n)f(j\pi/n).
	\end{equation}
	Note that the vector $\{R_n(t_0,j\pi/n)\}_{j=0}^{2n-1}$ can be calculated by the Fast Fourier Transform (FFT). In the next section, we will reduce the Laplace's and Helmholtz single layer potential to the evaluation of \eqref{logintegral}. 
	\begin{corollary}\label{coroll2}
		When $K(z)=\frac1{z-z_0}$, the modified trapezoidal rule yields
		\begin{equation}\label{reciprocal}
			I(Ku) \approx \begin{cases}
				T_{2n}\left(\frac{2-(z/z_0)^n-(z/z_0)^{n+1}}{2(z-z_0)}u\right) , & \quad |z_0|>1, \\
				\frac12T_{2n}\left(\frac{2-(z/z_0)^n-(z/z_0)^{n+1}}{2(z-z_0)}u + \frac{2-(z_0/z)^n-(z_0/z)^{n-1}}{2(z-z_0)}u\right) , & \quad |z_0|=1, \\
				T_{2n}\left(\frac{2-(z_0/z)^n-(z_0/z)^{n-1}}{2(z-z_0)}u\right) , & \quad |z_0|<1. \\
			\end{cases}
		\end{equation}
	\end{corollary}
	Such a kernel can be appeared in the gradient of Laplace's and Helmholtz single layer potentials.  When $|z_0|=1$, the left-hand side exists as Cauchy principal value, and a limit is needed to take on the right-hand side when $z_0 = e^{\pi ik/n}$ for some $k=1,\cdots,2n$. One can also make use of the formula \eqref{reciprocal} to construct the quadrature rule for kernel \eqref{power-law form} with $m\ge 1$ through differentiation of $K(z)=\frac1{z-z_0}$.
	\begin{corollary}\label{coroll3}
		When the kernel function is given in the form of 
		\begin{eqnarray}\label{doubelintegral}
			\frac{1}{1-e^{-i(t-t_0)}}+\frac{1}{1-e^{i(t-\overline{t_0})}}
		\end{eqnarray}
		with $t\in[0,2\pi]$ and $t_0\in\mathbb{C}$ close to the line segment $[0,2\pi]$. For $\Im t_0>0$, the modified trapezoidal rule yields
		\begin{equation}\label{inverse quadrature}
			\int_0^{2\pi}\left[\frac{1}{1-e^{-i(t-t_0)}}+\frac{1}{1-e^{i(t-\overline{t_0})}}\right]f(t)dt \approx \frac{\pi}{n}\sum_{j=0}^{2n-1}K_n(t_0,j\pi/n)f(j\pi/n),
		\end{equation}
		where
		\[K_n(t_0,t) = \frac{1-\frac12e^{-in(t-t_0)}-\frac12e^{-i(n+1)(t-t_0)}}{1-e^{-i(t-t_0)}}+\frac{1-\frac12e^{in(t-\overline{t_0})}-\frac12e^{i(n+1)(t-\overline{t_0})}}{1-e^{i(t-\overline{t_0})}}\]
		is $n$-th order Laurent expansion of \eqref{doubelintegral}. 
	\end{corollary}
	Corollary \eqref{coroll3} is essentially the same as Corollary \ref{coroll2}. Later in the paper, we will use it to compute the near field of double layer potential for Laplace's and Helmholtz equations.
	\section{Singularity swapping method}\label{singularity swapping}
	
	In the last section, we have proposed a modified trapezoidal rule to integrate the nearly singular integrals in the form of \eqref{nearly singular integral}. Compared to the original layer potential \eqref{layer potential}, there are two issues still remaining: the first is how to effectively find the singularity $z_0$,  and the second is how to easily construct the Laurent expansion for a given kernel function $K(x,\cdot)$ defined on a general smooth curve $\Gamma$. In this section, we introduce a variant of singularity swapping method \cite{Klinteberg2020} to address these two issues. For now we still assume $\Gamma$ is analytic, as the extension to piecewise analytic case will be discussed in the next section. 
	
	\subsection{Singularity cancellation}
	In this part, we mainly focus on the second issue and discuss the kernels given in the form of \eqref{logarithmic} and \eqref{power-law form}, more specifically, the single and double layer potentials for Laplace's equations, due to their wide applications in the potential theory. We expect the same idea can be extended to other kernels in a straightforward way. Identify $\mathbb R^2$ with $\mathbb C$ and assume that $\Gamma$ is parameterized by $g:\mathbb R\to\mathbb C$ that maps the interval $[0,2\pi)$ to $\Gamma$. Thus the parametric form of the single layer potential with logarithmic kernel \eqref{logarithmic} is
	\begin{equation}\label{parametric form}
		\int_0^{2\pi}\log|x-g(t)|^2f(t)dt, \quad x\in\mathbb C,
	\end{equation}
	where $f$ is an analytic function that consists of the density function and the Jacobian $|g'(t)|$. Bringing $x$ near $\Gamma$ introduces a singularity $t_0$ near real axis in the analytic continuation of $\log|x-g(t)|^2$. Specifically, it has a singularity when the complexified squared distance 
	\begin{eqnarray}\label{sqequ}
		|x-g(t)|^2 = (x_1-g_1(t))^2 + (x_2-g_2(t))^2	
	\end{eqnarray}
	goes to zero, which gives $x_1-g_1(t) = \pm i(x_2-g_2(t))$. Taking the negative case, the singularity $t_0$ thus obeys
	\begin{equation}\label{real equation of singularity}
		x_1+ix_2 = g_1(t_0)+ig_2(t_0),
	\end{equation}
	and one may check by Schwarz reflection that the positive case gives its conjugate $\bar t_0$.  Equation \eqref{real equation of singularity} can be rewritten as $x = g(t_0)$. The method to efficiently find $t_0$ will be discussed in the next subsection \ref{rootfinding}. For now we assume that $g:\mathbb C\to\mathbb C$ is 2$\pi$-periodic, analytic and single-valued in a sufficiently large strip $\{t\in\mathbb C:|\Im t|<\log r\}$, so the only solution to $x=g(t_0)$ is $t_0=g^{-1}(x)$, which is called the preimage of $x$ under $g$. 
	
	Now we rewrite the integral \eqref{parametric form} as
	\begin{equation}\label{singswap}
		\int_0^{2\pi}\log\left(\frac{|x-g(t)|^2}{4\sin\frac{t-t_0}{2}\sin\frac{t-\overline{t_0}}{2}}\right)f(t)dt + \int_0^{2\pi}\log\left(4\sin\frac{t-t_0}{2}\sin\frac{t-\overline{t_0}}{2}\right)f(t)dt,
	\end{equation}
	The first integral in \eqref{singswap} is analytic in a larger domain $\{t\in\mathbb C:|\Im t|<\log r\}$, so the classic trapezoidal rule \eqref{trapezoidal rule} can be used. The second integral can be numerically evaluated by the modified trapezoidal rule \eqref{log quadrature} in Corollary \ref{coroll1}. Hence, the singularities $t_0$ and $\overline{t_0}$ from $\log|x-g(t)|^2$ is swapped to the function $\log\left(4\sin\frac{t-t_0}{2}\sin\frac{t-\overline{t_0}}{2}\right)$ through \eqref{singswap}.
	
	For the Laplace double layer potential, it is given by
	\begin{equation}\label{laplace double layer potential}
		\int_0^{2\pi}\frac{(x-g(t))\cdot n(g(t))}{|x-g(t)|^2}|g'(t)|f(t)dt = \Im\int_0^{2\pi}\frac{g'(t)}{x-g(t)}f(t)dt,
	\end{equation}
	where $n(g(t))$ is unit outward normal vector of $\Gamma$ at $g(t)$. One way to swap its singularity is given by~\cite{Klinteberg2020} 
	\begin{equation}\label{Klinteberg method}
		\int_0^{2\pi}\frac1{1-e^{-i(t-t_0)}}\frac{(1-e^{-i(t-t_0)})g'(t)}{x-g(t)}f(t)dt.
	\end{equation}
	The second term $\frac{(1-e^{-i(t-t_0)})g'(t)}{x-g(t)}f(t)$ in \eqref{Klinteberg method} is smooth and plays the same role as $f(t)$ in \eqref{inverse quadrature}. We can therefore follow the quadrature rule given by \eqref{inverse quadrature} to discretize \eqref{Klinteberg method}. However, there is another approach to swap the singularity:
	\begin{equation}\label{laplace double layer potential swapped}
		\int_0^{2\pi}\left[\frac{g'(t)}{x-g(t)}-\frac{-i}{1-e^{-i(t-t_0)}}\right]f(t)dt + \int_0^{2\pi}\frac{-i}{1-e^{-i(t-t_0)}}f(t)dt,
	\end{equation}
	which we have found is numerically more accurate than \eqref{Klinteberg method}.
	
	One issue for the double layer evaluation is that when the target point $x$ is close to one of the sources $g(j\pi/n)$ with $0\leqslant j\leqslant 2n-1$, the accuracy is lost due to the singularity at $t=j\pi/n$. Similar phenomenon has also been observed in \cite{Klinteberg2020}. To address this issue, we apply the singularity subtraction technique by rewriting \eqref{laplace double layer potential} into 
	\begin{equation}\label{subtraction technique}
		\Im\int_0^{2\pi}\frac{g'(t)}{x-g(t)}(f(t)-f(j\pi/n))dt + \mu(x)f(j\pi/n),
	\end{equation}
	where
	\begin{equation}\label{mu}
		\mu(x) := \Im\int_0^{2\pi}\frac{g'(t)}{x-g(t)}dt = \begin{cases}
			-2\pi, & x\in D, \\
			-\pi, & x\in\Gamma, \\
			0, & x\in\mathbb R^2\backslash\overline{D},
		\end{cases}
	\end{equation}
	and $D$ is the bounded domain with boundary $\Gamma$. Then swap the singularity as in \eqref{laplace double layer potential swapped}. The new integrand vanishes at $t=j\pi/n$, so the rounding errors can be ignored.  There is no such simple subtraction techniques for single layer potentials \cite{2020arXiv200707965C}. However, it is also not necessary since the single layer potential suffers much less from cancellation errors, as one would see in the numerical experiments.
	
	As a comment on the extension of the techniques mentioned above to many other standard linear PDEs of interest (Helmholtz, Maxwell, Navier, etc), the integrals that we need to evaluate may be written as a combination of \eqref{parametric form} and \eqref{laplace double layer potential}. For example, for the Hemholtz single layer potential, the kernel $K(x,y)$ can be split by
	\begin{equation}\label{hemholtz single layer potential}
		\frac i4H_0^{(1)}(k|x-y|) = -\frac1{2\pi}J_0(k|x-y|)\log|x-y| + \text{smooth residue term},
	\end{equation}
	and for the Helmholtz double layer potential, the kernel $K(x,y)$ can be split by
	\begin{align}
		\frac{ik}4H_1^{(1)}(k|x-y|)\frac{(x-y)\cdot n(y)}{|x-y|} =& -\frac{k}{2\pi}\frac{J_1(k|x-y|)}{|x-y|}(x-y)\cdot n(y)\log|x-y| \notag\\
		&+ \frac1{2\pi}\frac{(x-y)\cdot n(y)}{|x-y|^2} + \text{smooth residue term}.\label{hemholtz double layer potential}
	\end{align}
	The first terms in both \eqref{hemholtz single layer potential} and \eqref{hemholtz double layer potential} have logarithmic type singularities and the second term in \eqref{hemholtz double layer potential} is the kernel of Laplace double layer potential. Therefore one may apply the same technique as Laplace to evaluate them. 
	
	\subsection{Quadrature method to find the pole}\label{rootfinding}
	Now the only missing ingredient in the scheme is how to find the singularity $t_0=g^{-1}(x)$ efficiently. Newton iteration is definitely an available option. However, it requires the evaluation of function values in the complex plane and the efficiency highly depends on the initial guess. Here we propose a quadrature method to find an approximation of $t_0$, which does not need any initial guess and in many situations is good enough to be directly applied in the singularity swapping method. 
	
	For the convenience of derivation, let $z_0=e^{it_0}$ and $g(t)=\gamma(e^{it})$, so that the equation $g(t_0)=x$ becomes $\gamma(z_0)=x$. Suppose $u$ is analytic in the annulus $r^{-1}\leqslant|z|\leqslant r$ except a simple pole $z_0$ with residue equal to $W$, where $z_0$ and $W$ are unknown. Similar to the analysis of equation \eqref{trapezoidal rule integral form}, it holds
	\begin{equation}\label{equation of pole 1}
		T_{2n}(z^{n}u) + \frac{2\pi Wz_0^{n-1}}{z_0^{2n}-1} = \int_{|z|=r}\frac{z^{n}u(z)}{z^{2n}-1}\frac{dz}{iz} - \int_{|z|=r^{-1}}\frac{z^{n}u(z)}{z^{2n}-1}\frac{dz}{iz} = O(r^{-n}),
	\end{equation}
	and 
	\begin{equation}\label{equation of pole 2}
		T_{2n}(z^{n+1}u) + \frac{2\pi Wz_0^{n}}{z_0^{2n}-1} = \int_{|z|=r}\frac{z^{n+1}u(z)}{z^{2n}-1}\frac{dz}{iz} - \int_{|z|=r^{-1}}\frac{z^{n+1}u(z)}{z^{2n}-1}\frac{dz}{iz} = O(r^{-n}).
	\end{equation}
	From the two formulas above, we can conclude that
	\begin{equation}\label{approximation of pole}
		\frac{T_{2n}(z^{n+1}u)}{T_{2n}(z^{n}u)}=z_0 + O(\max(|z_0|^n,|z_0|^{-n})r^{-n}),
	\end{equation}
	which provides a method to find the pole $z_0$ with high accuracy. The advantage of this formula is only function values on the unit circle are used and does not need any initial guess. To find an approximate solution to equation $\gamma(z_0)=x$, we can apply \eqref{approximation of pole} to $u(z)=\frac{\gamma'(z)}{\gamma(z)-x}$, which gives
	\begin{equation}\label{approximation of singularity}
		\hat z_0 := \frac{T_{2n}\left(\frac{z^{n+1}\gamma'(z)}{\gamma(z)-x}\right)}{T_{2n}\left(\frac{z^{n}\gamma'(z)}{\gamma(z)-x}\right)} = z_0 + O(\max(|z_0|,|z_0|^{-1})^nr^{-n}).
	\end{equation}

	In order to numerically test the accuracy of $\hat z_0$, we compare the performance of three different quadrature methods in the evaluation of nearly singular integrals. The first method is the classical trapezoidal rule given by \eqref{trapezoidal rule}, the second method is the interpolation quadrature given by \eqref{reciprocal} with exact $z_0$, and the third method is also the interpolation quadrature \eqref{reciprocal} but with $z_0$ replaced by $\hat z_0$. Figure \ref{test quadratures} tests the convergence of three integrals that are defined on $|z|=1$ and have the same nearly singular kernel $(z-1.1)^{-1}$. On top of the kernel, the three integrands are $z^{-10}+z^{10}$, $e^z$ and $(z-2)^{-1}$, respectively. For the first function $z^{-10}+z^{10}$, our analysis indicates that the interpolation quadrature \eqref{reciprocal} with $z_0=1.1$ is exact for $n\geqslant20$, which is clearly shown in Figure \ref{test quadratures}(a).  Using $\hat z_0$ only loses one algebraic accuracy. The second function $e^z$ is entire in $\mathbb{C}$. The interpolation quadrature achieves super-geometric convergence. Using $\hat z_0$ only causes the convergence speed to slow down slightly, as shown in \ref{test quadratures}(b). The third function $(z-2)^{-1}$ is analytic in a neighborhood of $|z|=1$ but not throughout the complex plane. By Theorem \ref{modifiedtrap}, the exact interpolation quadrature has a geometric convergence that is on the order of $2.2^{-n/2}$, which is confirmed by Figure \ref{test quadratures}(c). Here using $\hat z_0$ causes the convergence error to grow by a factor of $n$, which is still geometrically convergent. On the other hand, application of the classic trapezoidal rule converges very slow throughout all these three integrals. The results clearly show the interpolation quadrature can significantly improve the convergence rate, either using exact $z_0$ or the approximate value $\hat z_0$. 
	\begin{figure}[htbp]
		\centering
		\includegraphics[scale=0.9]{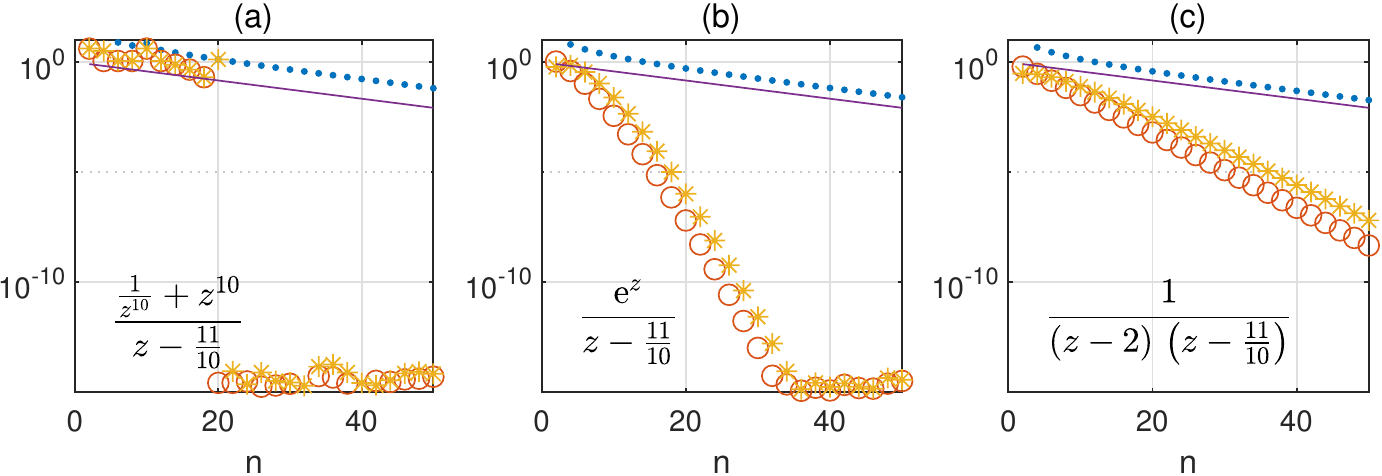}
		\caption{Relative error of classic trapezoidal rule (dots), interpolation quadrature given by \eqref{reciprocal} (circles) and interpolation quadrature \eqref{reciprocal} with $z_0$ replaced by $\hat z_0$ (stars) for three functions with the number of quadrature nodes $n$ ranging from 2 to 50. The solid curves is $1.1^{-n}$.}\label{test quadratures}
	\end{figure}
	
	To summarize, numerical observation indicates replacing $z_0$ with $\hat z_0$ into the quadrature formula  \eqref{reciprocal} gives a new residue term
	\begin{equation}\label{replace error}
		O(n\max(|z_0|,|z_0|^{-1})^{-n}r^{-n}),
	\end{equation} 
	which is slightly larger than the approximation error analyzed in \eqref{singquaderror}. We briefly analyze this term by taking $|z_0|>1$ as an example:
	\begin{align*}
		I(u) =& T_{2n}\left(u\right) - z_0^{-n}T_{2n}\left(\frac{1+z/z_0}{2}z^nu\right) + O(|z_0|^{-n}r^{-n}) \\
		=& T_{2n}\left(u\right) - \hat z_0^{-n}\left(1+O(|z_0|^nr^{-n})\right)^{-n}T_{2n}\left(\frac{1+z/z_0}{2}z^nu\right) + O(|z_0|^{-n}r^{-n}) \\
		=& T_{2n}\left(u\right) - \hat z_0^{-n}T_{2n}\left(\frac{1+z/\hat z_0}{2}z^nu\right) + O(n|z_0|^{-n}r^{-n}).
	\end{align*}
	The last identity holds because $T_{2n}\left(\frac{1+z/z_0}{2}z^nu\right)=O(|z_0|^{-n})$. In a word, the factor $n$ in \eqref{replace error} comes from the $n$-th power in the error term of \eqref{singquaderror}.
	
	If such an error is unacceptable, usually one Newton iteration
	\begin{equation}\label{newton iteration}
		\hat z_0 \leftarrow \hat z_0 - \frac{\gamma(\hat z_0)-x}{\gamma'(\hat z_0)}
	\end{equation}
	is enough to eliminate the error caused by singularity approximation. In order to improve the efficiency further, we can decrease the number of nodes in \eqref{approximation of singularity} and increase the number of Newton iterations. It is numerically found that for the same accuracy, one Newton iteration is roughly equivalent to reducing the number of nodes by half. By balancing the operations of the two parts, we can obtain a good approximation $\hat z_0$ within almost $O(n)$ operations.
	
	\section{Non-periodic nearly singular integrals}\label{section_nonperiodic_integrand}
	Now let us extend the interpolatory quadrature introduced in Section \ref{interpolatory quadrature} to non-periodic nearly singular integrals. Assume $\Gamma$ is a piecewise analytic closed or open curve, that may have finitely many corners. It is well known that when using integral  representation \eqref{layer potential} for Laplace's or Helmholtz equations with boundary $\Gamma$, corners or endpoints will introduce extra singularities on the density function $\phi$. It was shown in \cite{SERKH2016150} that the density function $\phi$ for boundary value problem of Laplace's equations has singularity of this type
	\begin{eqnarray}
		\phi(r) &=& Cr^{\alpha-1} + \text{smoother term}, \quad \text{for single layer potential}\label{corner-single}, \\
		\phi(r) &=& C_1+C_2r^{\alpha} + \text{smoother term}, \quad \text{for double layer potential}\label{corner-double},
	\end{eqnarray}
	where $r$ is the distance from the corner or endpoint of $\Gamma$, $C,C_1,C_2$ are constants. The exponent $\alpha=\frac{\pi}{\pi+|\pi-\theta|}$, where $\theta$ is the interior angle of the corner, and $\theta=0$ for open arc. This property also holds for Helmholtz equation up to a logarithmic factor~\cite{DensitySingular}. Because of the corner singularity, a uniform mesh yields a very poor convergence.
	
	Take out one smooth piece of $\Gamma$ and still denote it as $\Gamma$. We focus on calculating the layer potential generated by this piece. Let $\Gamma$ be parameterized by an analytic function $g:[0,2\pi]\to\mathbb C$ with $|g'(t)|$ nonzero for all $t\in[0,2\pi]$. Rewriting the layer potential \eqref{layer potential} based on the parametrization  yields $$\int_0^{2\pi}f(t)dt$$ with non-periodic integrand $f$, which is analytic in $(0,2\pi)$ but may have singularities at $t=0$ or $t=2\pi$. Suppose the interior angles at $g(0)$ and $g(2\pi)$ are $\theta_1$ and $\theta_2$ respectively and let
	\begin{equation}\label{alpha}
		\alpha:=\min\left(\frac{\pi}{\pi+|\pi-\theta_1|},\frac{\pi}{\pi+|\pi-\theta_2|}\right).
	\end{equation}
	From \eqref{corner-single}, we have $f(t):=-\frac1{2\pi}\log|g(t)-x||g'(t)|\phi(t)=O(t^{\alpha-1})$ for single layer potential. To take proper care of this corner singularity, we follow the idea in \cite{Cot98} by substituting a new variable so that the derivative of the new integrand vanishes up to a certain order at the endpoints. Then apply the singularity swapping method given in Section \ref{singularity swapping} to the transformed integrals.
	
	More specifically, let the function $w: [0,2\pi]\to[0,2\pi]$ be one-to-one, strictly monotonically increasing and analytic in $(0,2\pi)$.  Assume that $w$ at the two endpoints vanishes up to order $p\in\mathbb N$, i.e., $w(t)=O(t^p)$ near $t=0$ and $w(t)=2\pi+O((t-2\pi)^p)$ near $t=2\pi$. We then substitute $t\to w(t)$ to obtain
	\begin{equation}\label{substituted integral}
		\int_0^{2\pi}f(t)dt = \int_0^{2\pi}w'(t)f(w(t))dt.
	\end{equation}
	Applying the trapezoidal rule to the transformed integral now yields the quadrature formula
	\begin{equation}\label{substituted trapezoidal rule}
		\int_0^{2\pi}f(t)dt \approx \frac{2\pi} n\sum_{j=0}^{n}{}''w'(2\pi j/n)f(w(2\pi j/n)),
	\end{equation}
	where $\sum{}''$ means that the first and last term need to be halved. A typical example for such a substitution is given by
	\begin{equation}\label{asin transform}
		\begin{aligned}
			w(t) =& \frac{(p-2)!!}{(p-3)!!}\int_0^t(t-\tau)(\sin \tau)^{p-2}d\tau \\
			=& t - \left(\sin t + \frac{(\sin t)^3}{6} + \frac{3(\sin t)^5}{40} + \cdots + \frac{(p-4)!!}{(p-3)!!}\frac{(\sin t)^{p-2}}{p-2}\right)\\
			=& t- I_p(\sin t), \quad\text{if $p\geqslant3$ is odd,}
		\end{aligned}
	\end{equation}
	where $I_p(\cdot)$ in \eqref{asin transform} is the truncated Taylor series of $\arcsin$ up to order $p-2$. There are many other transformations that can serve the purpose, and readers are referred to \cite{Elliott1998SigmoidalTA} for more details. We choose \eqref{asin transform} simply because its quadrature points are relatively uniformly distributed.
	
	The following theorem provides an error estimate for \eqref{substituted trapezoidal rule}. It can be derived from the Euler-Maclaurin formula, but here we provide an alternative proof using the residue theorem.
	\begin{theorem}\label{error of nonperiod trapezoidal rule}
		Assume that $u$ is analytic in the annulus $r^{-1}\leqslant|z|\leqslant r$ cut along the interval $[r^{-1},r]$ for some $r>1$. Assume furthur that $u(e^{it})=u(1)+C_3t^{\alpha-1}+o(t^{\alpha-1})$ as $t\to0^+$ and $u(e^{it})=u(e^{2\pi i})+C_4(2\pi-t)^{\alpha-1}+o(2\pi-t)^{\alpha-1}$ as $t\to2\pi^-$ for some $\alpha>1$ and constants $C_3,C_4$. Then
		\begin{equation}\label{singular interpolatory quadrature error asymptotic form}
			T_n(u)-I(u) = (C_3+C_4)\cos\left(\frac\pi2\alpha\right)\Gamma(\alpha)\zeta(\alpha)n^{-\alpha} + o(n^{-\alpha}), \quad n\to\infty,
		\end{equation}
		where $I(u)$ is given by \eqref{integral} and $T_n(u)$ is the non-periodic version of trapezoidal rule
		\[T_n(u) = \frac{2\pi}{n}\sum_{k=0}^n{}''u(e^{2\pi ik/n}).\]
	\end{theorem}
	\begin{proof}
		Since $u$ has a branch cut along $[r^{-1},r]$, equation \eqref{trapezoidal rule error integral form} becomes
		\begin{equation}\label{singular trapezoidal rule error integral form}
			T_n(u)-I(u) = \int_{\Gamma_1}\frac{u(z)}{z^n-1}\frac{dz}{iz} + \int_{\Gamma_2}\frac{u(z)}{1-z^{-n}}\frac{dz}{iz}.
		\end{equation}
		where $\Gamma_1$ and $\Gamma_2$ are two contours shown in Figure  \ref{contour}. The integral in \eqref{singular trapezoidal rule error integral form} near $z=1$ exists as Cauchy principal value.
		\begin{figure}[ht]
			\centering
			\includegraphics[scale=0.8]{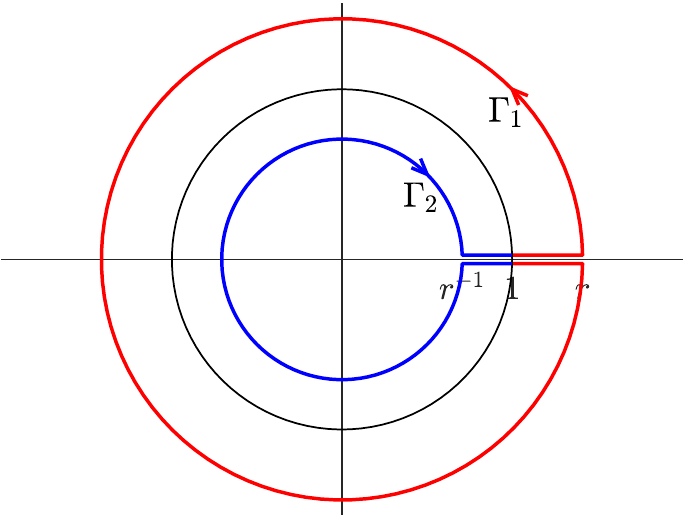}
			\caption{$\Gamma_1\cup \Gamma_2$ is the boundary of $\exp((-\log r,\log r)+(0,2\pi)i)$. $\Gamma_1$ is in red and traversed in the counterclockwise direction. $\Gamma_2$ is in blue and traversed in the clockwise direction.}\label{contour}
		\end{figure}
		Since the integral  near $z=1$  is dominant in \eqref{singular trapezoidal rule error integral form}, it holds the following estimate on the upper edge of the branch cut $[r^{-1},r]$
		\begin{align*}
			& \int_{r^{-1}}^1\frac{u(z)}{1-z^{-n}}\frac{dz}{iz} + \int_{1}^r\frac{u(z)}{z^n-1}\frac{dz}{iz} \\
			=& \int_{1-n^{-1/2}}^1\frac{u(z)}{1-z^{-n}}\frac{dz}{iz} + \int_{1}^{1+n^{-1/2}}\frac{u(z)}{z^n-1}\frac{dz}{iz} + O(e^{-n^{1/2}}) \\
			=& \int_{1-n^{-1/2}}^1\frac{u(1)+C_3(-i\log z)^{\alpha-1}}{1-z^{-n}}\frac{dz}{iz} \\&+  \int_{1}^{1+n^{-1/2}}\frac{u(1)+C_3(-i\log z)^{\alpha-1}}{z^n-1}\frac{dz}{iz} + o(n^{-\alpha}) \\
			=& C_3\int_{1-n^{-1/2}}^1\frac{(-i\log z)^{\alpha-1}}{1-z^{-n}}\frac{dz}{iz} + C_3\int_{1}^{1+n^{-1/2}}\frac{(-i\log z)^{\alpha-1}}{z^n-1}\frac{dz}{iz} + o(n^{-\alpha}) \\
			=& C_3\int_{0}^1\frac{(-i\log z)^{\alpha-1}}{1-z^{-n}}\frac{dz}{iz} + C_3\int_{1}^{\infty}\frac{(-i\log z)^{\alpha-1}}{z^n-1}\frac{dz}{iz} + o(n^{-\alpha}) \\
			=& C_3\int_0^\infty\frac{(ix)^{\alpha-1}}{1-e^{nx}}\frac{dx}{i} + C_3\int_0^\infty\frac{(-ix)^{\alpha-1}}{e^{nx}-1}\frac{dx}{i} + o(n^{-\alpha}) \\
			=& C_3\cos\left(\frac\pi2\alpha\right)\int_0^\infty\frac{x^{\alpha-1}}{e^{nx}-1}dx + o(n^{-\alpha}) \\
			=& C_3\cos\left(\frac\pi2\alpha\right)\Gamma(\alpha)\zeta(\alpha)n^{-\alpha} + o(n^{-\alpha}),
		\end{align*}
		where $\Gamma(\cdot)$ is the Gamma function, and $\zeta(\cdot)$ is the Riemann-zeta function. Since $(1-n^{-1/2})^n=O(e^{-n^{1/2}})$ as $n\to\infty$, we are free to add or remove integrals other than $[1-n^{-1/2},1+n^{-1/2}]$. Same estimate holds for the integral on the lower edge of the branch cut $[r^{-1},r]$. The conclusion follows from the two  estimates combined together.
	\end{proof}
	
	Assume that $f=O((t(2\pi-t))^{\alpha-1})$ near the two endpoints for some $\alpha>0$.  Since the transformed integrand  $w'(t)f(w(t))=O((t(2\pi-t))^{\alpha p-1})$, from Theorem \ref{error of nonperiod trapezoidal rule}, the error of \eqref{substituted trapezoidal rule} is $O(n^{-\alpha p})$. By choosing $p$ large enough, theoretically we can obtain high order convergence. However, too large $p$ is numerically unstable due to the rounding errors in float-point arithmetic. In practice, $p=5$ or $7$ would be sufficient. 
	
	When the target point $x$ is very close to $\Gamma$, one needs to use the modified trapezoidal rule \eqref{modifedtrap} after the variable substitution. In the non-periodic case, it will introduce an additional error term compared to \eqref{singquaderror}, as shown in the following theorem.
	
	\begin{theorem}\label{thm4}
		Under the assumptions of Theorem \ref{error of nonperiod trapezoidal rule}, assume further that $K$ is analytic on the unit circle $|z|=1$. When $n\to\infty$, it holds
		\begin{align}\label{correction}
			&T_{2n}(K_nu)-I(Ku)  \notag \\ =&\int_{r^{-1}}^r\frac{z_-^nK(z)-z^{-n}K_n^+(z)-z^{n}K_n^-(z)}{z^{n}-z^{-n}}\left(u(1)-u(e^{2\pi i})\right)\frac{dz}{iz} + O(n^{-\alpha}), 
		\end{align}
		where $K_n,K_n^+,K_n^-$ are the same as in Section \ref{subsection nearly singular integral} and 
		\[z_-^n = \begin{cases}
			z^{-n}, & |z|\geqslant 1, \\
			z^n, & |z|<1.
		\end{cases}\]
	\end{theorem}
	\begin{proof}
		From equation \eqref{interpolatory quadrature error integral form}, it holds
		\begin{align}
			& T_{2n}(K_nu)-I(Ku) \notag\\
			=& \int_{\Gamma_1}\left(\frac{K_n(z)}{z^{2n}-1}-K_n^-(z)\right)u(z)\frac{dz}{iz} - \int_{\Gamma_2}\left(\frac{K_n(z)}{z^{-2n}-1}-K_n^+(z)\right)u(z)\frac{dz}{iz} \notag\\
			=& \int_{\Gamma_1}\frac{K(z)-K_n^+(z)-z^{2n}K_n^-(z)}{z^{2n}-1}u(z)\frac{dz}{iz} - \int_{\Gamma_2}\frac{K(z)-z^{-2n}K_n^+(z)-K_n^-(z)}{z^{-2n}-1}u(z)\frac{dz}{iz} \notag\\
			=& \int_{\Gamma_1}\frac{z^{-n}K(z)-z^{-n}K_n^+(z)-z^{n}K_n^-(z)}{z^{n}-z^{-n}}u(z)\frac{dz}{iz} \notag \\&- \int_{\Gamma_2}\frac{z^nK(z)-z^{-n}K_n^+(z)-z^nK_n^-(z)}{z^{-n}-z^n}u(z)\frac{dz}{iz} \notag\\
			=& \int_{\Gamma_1\cup\Gamma_2}\frac{z_-^nK(z)-z^{-n}K_n^+(z)-z^{n}K_n^-(z)}{z^{n}-z^{-n}}u(z)\frac{dz}{iz} \notag\\
			=& \int_{r^{-1}}^r\frac{z_-^nK(z)-z^{-n}K_n^+(z)-z^{n}K_n^-(z)}{z^{n}-z^{-n}}\left(u(1)-u(e^{2\pi i})\right)\frac{dz}{iz} + O(n^{-\alpha}). \notag 
		\end{align}
	\end{proof}
	
	\begin{remark}
		For the first error term in \eqref{correction}, the integral interval can be replaced by any interval $[a,b]$ with $0\leqslant a<1$ and $1<b\leqslant\infty$, because the integrand decays exponentially when $z$ is away from 1. In most cases, $u(1)=u(e^{2\pi i})=0$, so there's no need to calculate this integral. In the remaining cases, it is easy to calculate the error term numerically and subtract it from the right hand side of equation \eqref{correction}. Our idea is to first replace the interval to $[e^{-1},e]$, substitute $z=e^{it}$ and then apply $2n+1$ points trapezoidal rule in $[-i,i]$, or shift to $[\pi-i,\pi+i]$ if the singularity $t_0$ is close to the imaginary axis. More specifically, if we use \eqref{laplace double layer potential swapped} to swap the singularity, the corresponding integral in \eqref{correction} is (the constant multiplier is ignored)
		\begin{align*}
			& \int_{- i}^{ i}\frac{e^{int}\frac{e^{-in(t-t_0)}+e^{-i(n+1)(t-t_0)}}{1-e^{-i(t-t_0)}}+e^{-int}\frac{e^{in(t-\overline{t_0})}+e^{i(n+1)(t-\overline{t_0})}}{1-e^{i(t-\overline{t_0})}}}{e^{int}-e^{-int}}dt \\
			=& \int_{- i}^{i}\frac{e^{int_0}\left(\cot\frac{t-t_0}{2}-\cot\frac{-t_0}{2}\right)-e^{-in\overline{t_0}}\left(\cot\frac{t-\overline{t_0}}{2}-\cot\frac{-\overline{t_0}}{2}\right)}{2\sin(nt)}dt \\
			:=& \int_{- i}^{ i}\frac{M_n(t_0,t)}{2\sin(nt)}dt.
		\end{align*}
		We approximate this integral by
		\[ \begin{dcases}
			\frac{i}{n}\sum_{k=-n}^n{}''\frac{M_n(t_0,\pi+ik/n)}{2\sin(ik)} + 2\pi i\sum_{k=n}^{2n}{}''\frac{M_n(t_0,k\pi/n)}{2(-1)^kn}, & \mbox{ if }\Re t_0\in[0,\frac{\pi}{2}), \\
			\frac{i}{n}\sum_{k=-n}^n{}''\frac{M_n(t_0,ik/n)}{2\sin(ik)}, &  \mbox{ if } \Re t_0\in[\frac{\pi}{2},\frac{3\pi}{2}], \\
			\frac{i}{n}\sum_{k=-n}^n{}''\frac{M_n(t_0,\pi+ik/n)}{2\sin(ik)} + 2\pi i\sum_{k=0}^{n}{}''\frac{M_n(t_0,k\pi/n)}{2(-1)^kn}, &  \mbox{ if } \Re t_0\in(\frac{3\pi}{2},2\pi),
		\end{dcases}\]
		where we use the residue theorem first and then apply the trapezoidal rule for $\Re t_0\in[0,\frac{\pi}{2})\cup(\frac{3\pi}{2},2\pi]$,  and use the trapezoidal rule only for $\Re t_0\in[\frac{\pi}{2},\frac{3\pi}{2}]$.
	\end{remark}
	
	From Theorem \ref{thm4}, although $T_{2n}(K_nu)$ cannot essentially improve the convergence rate, it will reduce the error caused by the near singularity, which is the dominant part when $n$ is small, and make the error decay uniformly as $O(n^{-\alpha p})$ for arbitrarily close targets.
	
	To illustrate the performance of $T_{2n}(K_n u)$, we again test three integrals using the three quadrature methods as mentioned in subsection \ref{rootfinding}. As before, we use the root finding method to find an approximate pole $\hat z_0$ of $K$ and apply the singularity swapping method introduced in Section \ref{singularity swapping}. All the three integrals have the same nearly singular kernel $(z-1.1)^{-1}$. The three density functions are $e^{-(z+1)^{-2}}$, $(1/\sqrt z+\sqrt z)^3$ and $1/\sqrt z+\sqrt z$, respectively. The first one $e^{-(z+1)^{-2}}$ is $C^\infty$ on $|z|=1$ but not analytic.  The second function $(1/\sqrt z+\sqrt z)^3$ is $C^2$ and the third function $1/\sqrt z+\sqrt z$ is just continuous. Figure \ref{test quadratures oscillatory}(a) shows for the first function all three quadratures are essentially sub-geometrically convergent. Figure \ref{test quadratures oscillatory}(b) and (c) show the convergence rates for the second and third functions are both algebraic, which is consistent with the error analysis given by equation \eqref{singular interpolatory quadrature error asymptotic form}. Throughout all the examples, using $\hat z_0$ causes the convergence speed to be slightly slower than using the exact $z_0$, but it is still much faster than the classic trapezoidal rule.
	
	\begin{figure}[htbp]
		\centering
		\includegraphics[scale=0.9]{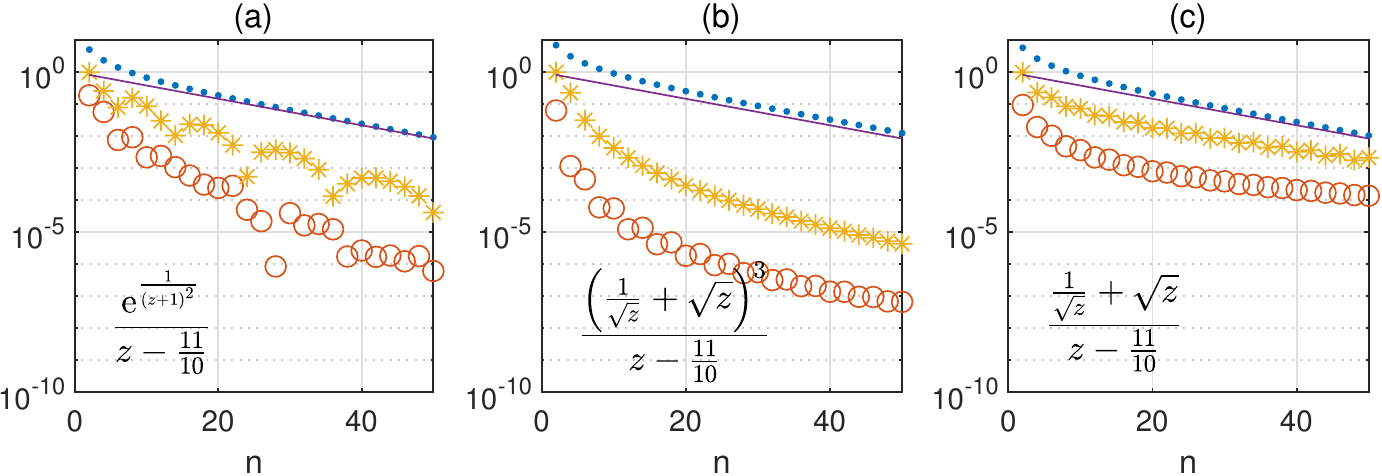}
		\caption{Relative error of trapezoidal rule (dots), interpolation quadrature given by \eqref{reciprocal} (circles) and interpolation quadrature \eqref{reciprocal} with $z_0$ replaced by $\hat z_0$ give by equation \eqref{approximation of pole} (stars) for three functions with the number of quadrature nodes $n$ ranging from 2 to 50. The solid curves is $1.1^{-n}$.}\label{test quadratures oscillatory}
	\end{figure}
	
	\section{Numerical examples}\label{numerical examples}
	In this section, we test the modified trapezoidal rule with singularity swapping method for the near field evaluation in several examples. To verify the accuracy, we construct an artificial solution for Laplace's and Helmholtz equations by the point source technique. Namely, for the interior problem, assuming the computational domain is $\Omega$, we place a point source outside of $\Omega$ and impose the exact boundary condition on $\partial \Omega$. When the boundary integral equation is solved correctly, the point source solution can be recovered in $\Omega$. Similar technique applies to the exterior problem. Details are given in \cite{JL2014}. All the experiments are carried out by MATLAB on a laptop with an Intel CPU inside.  
	\subsection{Laplace interior Dirichlet problem in a star-shaped domain}\label{laplacenum}
	In this example, we test the Laplace's equation $\Delta u=0$ in the interior of a star-shaped domain $D$ with boundary parameterized  by $g(t)=(1+0.3\cos(5t))e^{it},\ t\in[0,2\pi]$. Impose the Dirichlet boundary conditions $u=u_0$ on $\partial D$ with $u_0(x)=\log|x-x_0|$ and $x_0=3+3i$. We discretize $\partial D$  with equal spaced points in the parameter space $[0,2\pi]$. The solution to the Laplace's equation is represented by either single or double layer potential, which results in two different integral equations in $\phi$. They can be solved using the Nystr\"om method \cite{Kress2010} and the proposed quadrature. Once $\phi$ is found, we evaluate the layer potential in $D$ based on the value of $\phi$ at the discretization points, and compare the numerical solution with the exact solution $u_0$.  For close evaluation, we use \eqref{singswap} to evaluate the single layer potential and \eqref{laplace double layer potential swapped} for the double layer potential.
	
	In our first test, we solve the integral equation with $n=64$ and evaluate the field on a uniform grid that covers the domain with spacing $0.01$.  Results of relative error $E_{rel}$ for both single and double layer potentials are shown in Figure \ref{error}. The blank point indicates that the calculated result of that point is ``exact" in float point arithmetic. Far from the boundary,  both trapezoidal rule and the singularity swapping method give the same accuracy. Near the boundary, trapezoidal rule gets no digit whereas the proposed quadrature method still gives at least 10 digits accuracy.
	\begin{figure}[ht]
		\centering
		\includegraphics[scale=0.8]{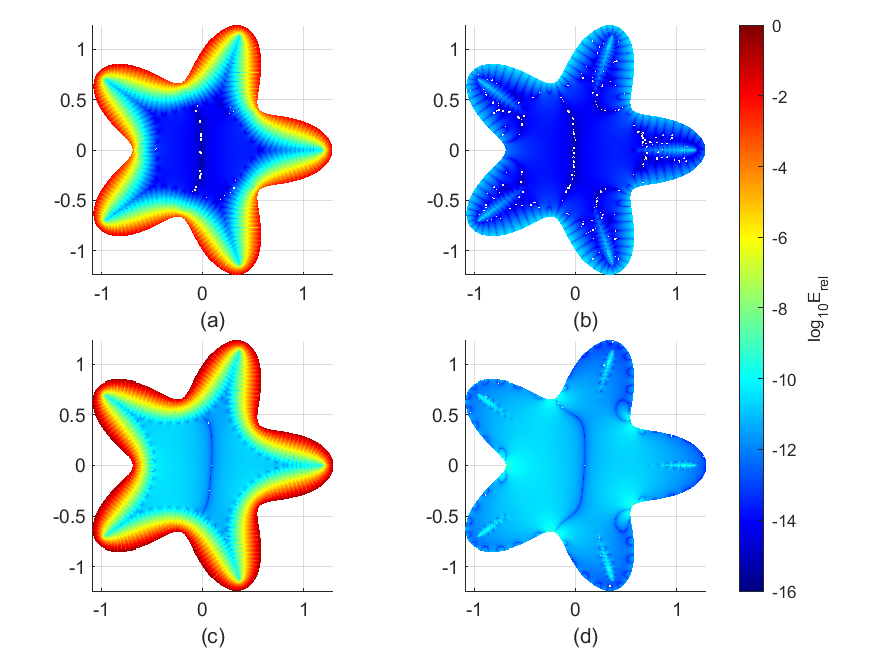}
		\caption{Near field evaluation for Laplace equation. (a) and (c): Results based on classic trapezoidal rule with (a) for single layer potential and (c) for double layer potential. (b) and (d): Results based on the modified trapezoidal rule with (b) for single layer potential and (d) for double layer potential.}\label{error}
	\end{figure}
	
	In the second test, we solve the same problem using $n=8,16,\cdots,128$ and evaluate the solution at a specific target point $x=0.5+i$ to test the convergence rate. The target point $x$ is very close to the boundary. Its nearest two preimages $g^{-1}(x)$ are $t_1=1.058224887371462 + 0.045168525183462i$ and $t_2=0.422785911369020 - 0.293359723744667i$. The preimage of $x$ under the parameter $t$ and the parameter $z_0$ on the unit circle is related by $z_0=e^{it}$. From \eqref{interpolatory quadrature error asymptotic form}, the convergence rate is $O(e^{-n(|\Im t_1|+|\Im t_2|)})$. The error of single layer case is smaller by a factor of $n$ because the estimate \eqref{cauchy's estimate} can be improved to $|c_n|=O(n^{-1}|z_0|^{-n})$ for $K(z)=\log(z-z_0)$. Numerical results given in Figure \ref{convergence} confirm the exponential decay of the error.
	\begin{figure}[ht]
		\centering
		\includegraphics[scale=0.8]{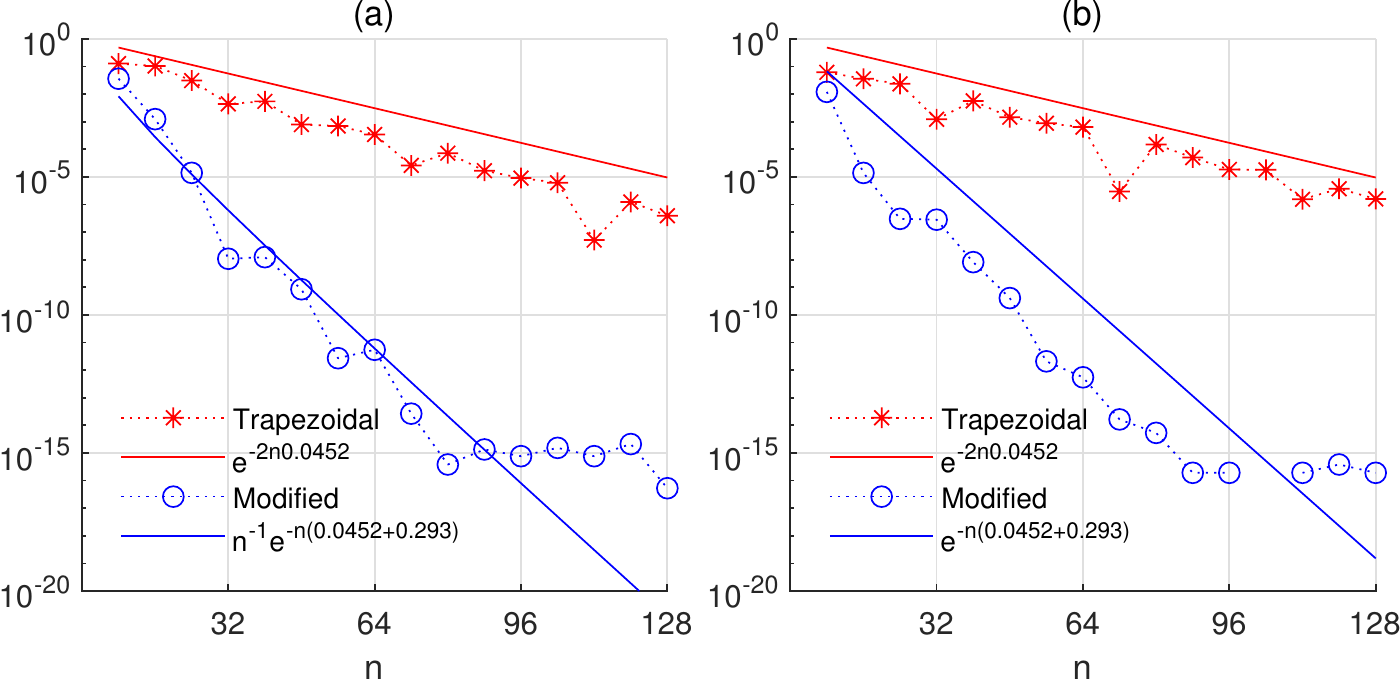}
		\caption{Relative error of evaluating $u(x)$ for Laplace's equation using trapezoidal rule (stars) and modified trapezoidal (circles) and their theoretical convergence rate (solid curves). The red and blue lines plot the predicted convergence rate. (a) Single layer potential. (b) Double layer potential.}\label{convergence}.
	\end{figure}
	
	In the third test, we solve the same problem with $n=128$ and evaluate the field on a slice of $D$ defined by the mapping $g$ of the square $[1.66\pi,1.76\pi]\times[10^{-8},0.15]$ in $\mathbb C$, i.e.,
	\[D = \{g(s): \Re s\in[1.66\pi,1.76\pi], \Im s\in[10^{-8},0.15]\},\]
	which is extremely close to the boundary. We use $300\times300$ grid points that are uniform in $\Re s$ and logarithmic in $\Im s$. The results of this test are shown in Figure \ref{extremely}. One can see that the single layer potential keeps 13 digits accuracy except at $\Re s=\frac{224}n\pi\approx5.5$. However, for the double layer potential, Figure \ref{extremely}(b) shows the accuracy around each source drops to 10 digits. Using the subtraction technique, the accuracy improves to 14 digits, as shown in Figure \ref{extremely}(c), which implies the subtraction technique is every effective for the field evaluation near the source.
	\begin{figure}[ht]
		\centering
		\includegraphics[scale=0.8]{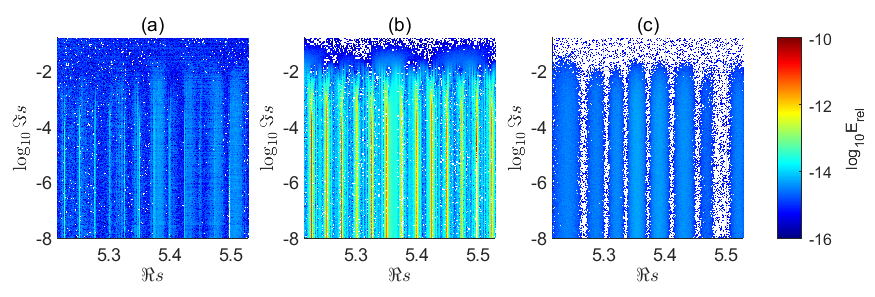}
		\caption{Test on the extremely close field evaluation for Laplace's equation. (a) Single layer potential without substraction technique. (b) Double layer potential without substraction technique. (c) Double layer potential with substraction technique.}  \label{extremely}
	\end{figure}
	
	\subsection{Helmholtz interior Dirichlet problem in an inkblot-shaped domain}\label{helmnum}
	In this example, we test our method on a piecewise analytic curve by solving an exterior Dirichlet problem of Helmholtz equation
	\[\begin{cases}
		\Delta u+\kappa^2 u = 0, & \text{in } \mathbb R^2\backslash\overline D, \\
		u = \frac{i}4H_0^{(1)}(\kappa|x-x_0|), & \text{on } \partial D, \\
		\frac{\partial u}{\partial r} - i\kappa u = o(r^{-\frac12}), & r=|x|\to\infty,
	\end{cases}\]
	with $\kappa=3$, $x_0 = 1+i$, and the boundary $\partial D$ parameterized by
	\[g(t) = \left (4+2|\cos(4t)|\sin(4t)\right )e^{it}, \quad t\in[0,2\pi].\]
	It is clear that the exact solution is $u_0(x):=\frac{i}4H_0^{(1)}(\kappa|x-x_0|)$. The contour $\partial D$ has eight corners with interior angle $\arctan\frac43$ or $2\pi-\arctan\frac43$, and the corresponding $\alpha$ in \eqref{alpha} is
	\[\alpha = \frac{1}{2-\frac1\pi\arctan\frac43} = 0.586567797561351\dots\]
	
	The contour $\partial D$ is divided into eight analytic pieces. For each piece, the parameter is shifted and rescaled to $[0,2\pi]$ and then substituted by $w(t)$ given by \eqref{asin transform} with $p=7$. Each piece is discretized with $2n$ points uniformly in $t$. Similar to the first example, solution to the Helmholtz equation is represented by either single or double layer potential.
	
	The first test is shown in Figure \ref{inkerror}. This example evaluates the field with $n=32$ on a uniform grid that covers $[-6,6]^2$ with spacing $0.1$. Far from the boundary, both trapezoidal rule and singularity swapping method give 12 digits. On the other hand, trapezoidal rule gets no digit near the boundary, whereas the proposed quadrature gets 8 digits accuracy in most of the near boundary regions and drops to 6 digits in the concave part of the boundary.
	\begin{figure}[ht]
		\centering
		\includegraphics[scale=0.8]{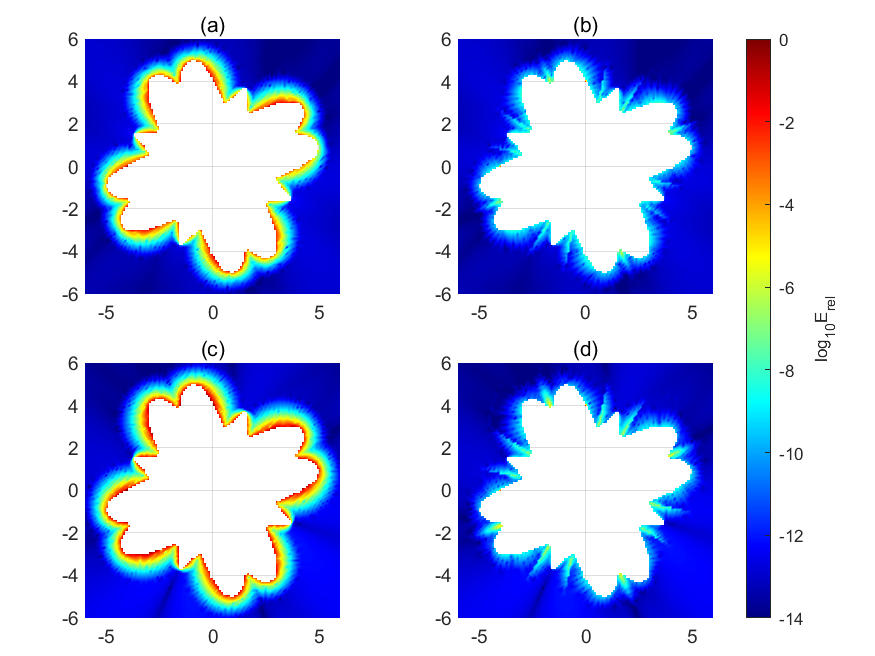}
		\caption{Relative error for the close field evaluation of Helmholtz equation.  (a) and (c): Results based on classic trapezoidal rule with (a) for single layer potential and (c) for double layer potential. (b) and (d): Results based on the modified trapezoidal rule with (b) for single layer potential and (d) for double layer potential.}\label{inkerror}
	\end{figure}
	
	The second test is shown in Figure \ref{inksingular}, this example tests the singular behavior of $\phi$ near a cornor $g(\frac\pi8)$ with $n=64$. The result confirms the estimates in \eqref{corner-single} and \eqref{corner-double},  which shows the solution $\phi$ is very accurate through our quadrature method.
	\begin{figure}[ht]
		\centering
		\includegraphics[scale=0.8]{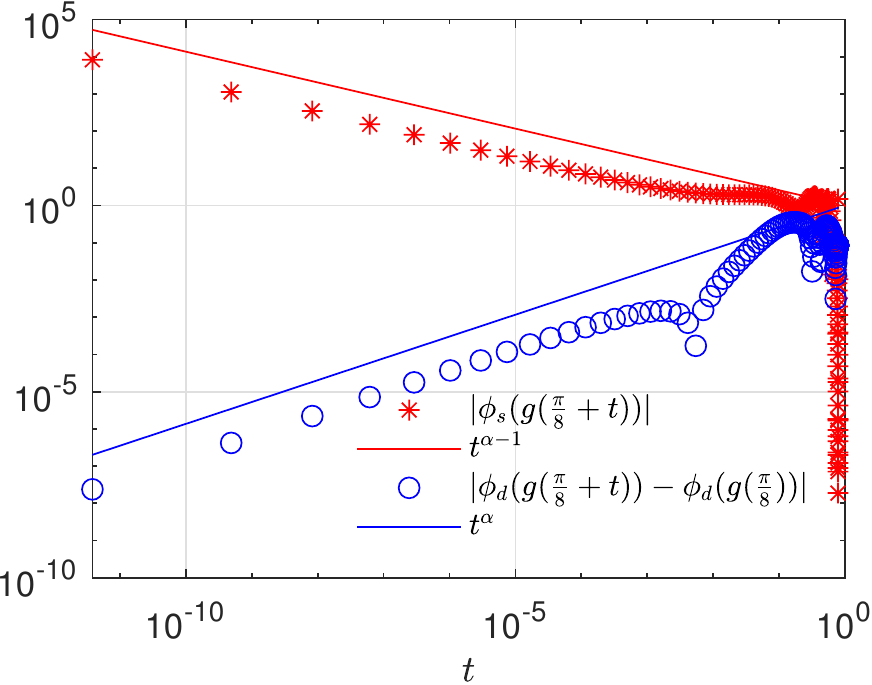}
		\caption{Singular behavior of $\phi$ near the corner. $\phi_s$ (stars): density obtained by the single layer representation. $\phi_d$ (circles): density obtained by the double layer representation. The solid curves is the theoretical behavior. Details are given in  Example \ref{helmnum}.}\label{inksingular}
	\end{figure}
	
	In  the third test, we solve the same problem using $n=4,8,\cdots,64$ and evaluate the field at a specific target point $x=3+3i$ to test the convergence rate. The target point $x$ is very close to the boundary. Its nearest two preimages of $x$ under transformed parameterization are
	$3.047236160485154 - 0.050219779775915i$ and $4.428039970225041 - 0.649002137868717i$. The results of this test are shown in Figure \ref{inkconvergence}. When $n\leqslant36$, the convergence rate of proposed method is exponential and mainly govern by the second preimage. When $n\geqslant40$, the algebraic term $n^{-\alpha p}$ becomes the dominant part of the error.
	\begin{figure}[ht]
		\centering
		\includegraphics[scale=0.8]{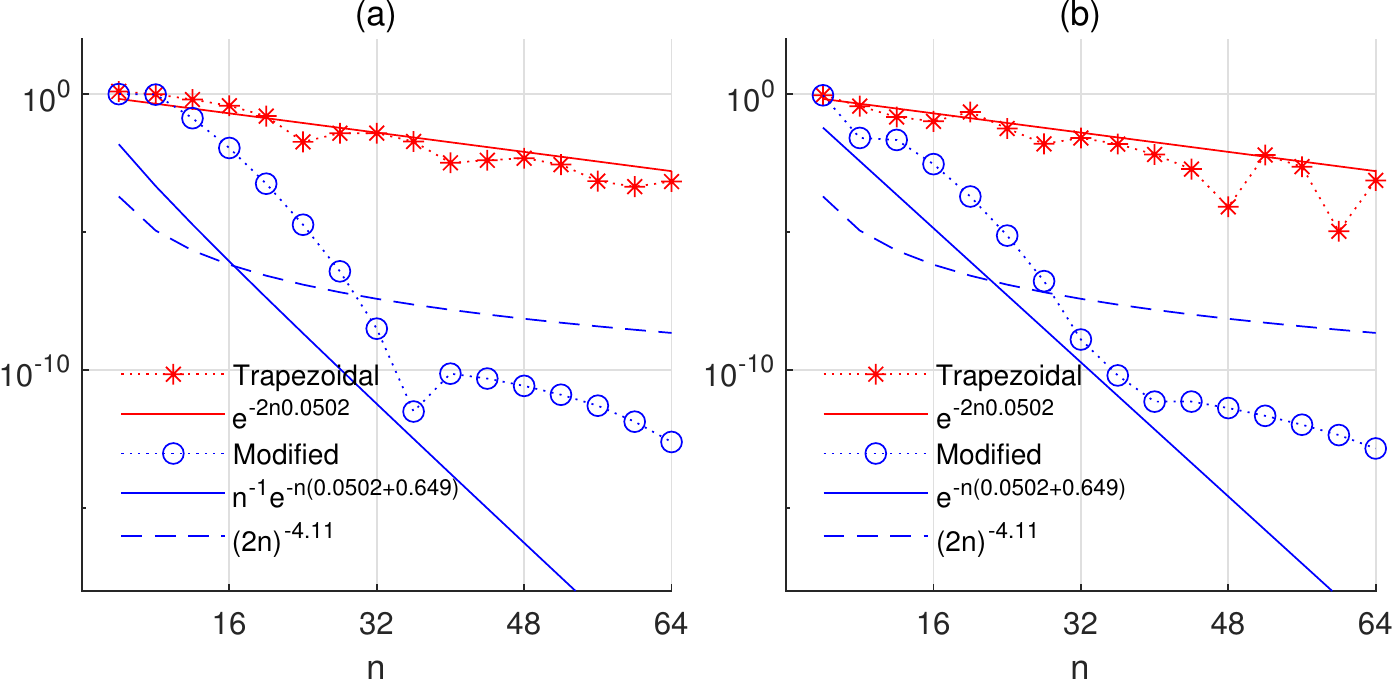}
		\caption{Relative error of evaluating $u(x)$ of Helmholtz equation using trapezoidal rule (stars) and modified trapezoidal (circles) and their theoretical convergence rate (solid curves). The red and blue lines plot the predicted convergence rate. (a) Single layer potential. (b) Double layer potential.}\label{inkconvergence}
	\end{figure}
	
	In the fourth test, we solve the same problem using $n=64$ and evaluate the field in a sector defined by
	\[\left\{g\left(\frac\pi8\right)+re^{i\theta}: r\in[10^{-8},1], \theta\in\left[\frac\pi8-\arctan\frac12,\frac\pi8+\arctan\frac12\right]\right\},\]
	which is extremely close to the corner. We use a $100\times100$ grid that is uniform in $\theta$ and logarithmic in $r$. The results of this test are shown in Figure \ref{inkextremely}. The single layer representation gives 7 digits accuracy whereas the double layer potential only gives 4 digits accuracy but can be improved to 9 digits with subtraction techniques. Compared to the field evaluation that are far away from the corner, several digits are lost for field evaluation in the neighborhood of corners. The loss can be attributed to the large value of the intergrand on the interval $[r^{-1}, r]$ in \eqref{singular trapezoidal rule error integral form}. However, it is still much better than the classic method.
	
	\begin{figure}[ht]
		\centering
		\includegraphics[scale=0.8]{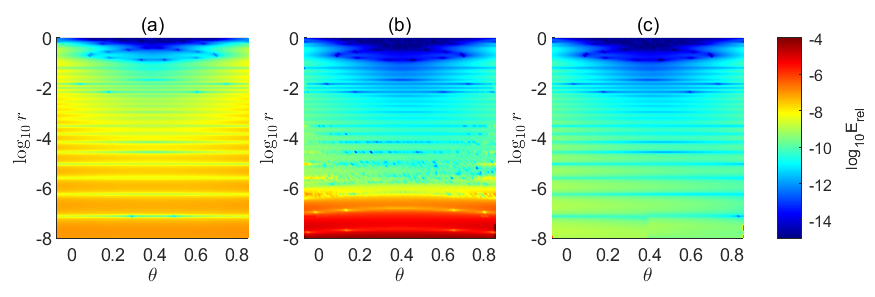}
		\caption{Test on the extremely close corner field evaluation for Helmholtz equation. (a) Single layer potential without substraction technique. (b) Double layer potential without substraction technique. (c) Double layer potential with substraction technique.}\label{inkextremely}
	\end{figure}
	
	\section{Conclusion}\label{conclusions}
	We present a new version of the singularity swapping method for the close field evaluation of two dimensional layer potentials. Our approach combines the classic trapezoidal rule and the analytic expansion of singular kernels.  The resulted quadrature formula is easy to use and achieves exponential convergence. Rigorous error analysis is provided based on the Cauchy integral and residue theorem. The method is also extended to piecewise analytic curves by a substitution of variables and achieves a high order convergence rate. Numerical results demonstrate the effectiveness of the proposed method in both smooth and non-smooth geometries.
	
	We plan to extend our work on several aspects. Our immediate next step is to consider the application to other elliptic PDE kernels, such as Navier equations in elasticity and Stokes equations in fluid mechanics. We also plan to develop fast boundary integral solvers for multiple wave scattering problems based on the proposed quadrature, and generalize our work to three dimensional close evaluation problems.

\end{document}